\documentclass[12pt]{amsart}
\usepackage{amsmath, amsthm, amscd, amsfonts}

\newtheorem{theorem}{Theorem}[section]
\newtheorem{lemma}[theorem]{Lemma}
\newtheorem{proposition}[theorem]{Proposition}
\newtheorem{corollary}[theorem]{Corollary}
\theoremstyle{definition}
\newtheorem{definition}[theorem]{Definition}

\newtheorem{problem}[theorem]{Problem}
\theoremstyle{remark}

\newtheorem{remark}[theorem]{Remark}
\numberwithin{equation}{section}

\newfont{\kh}{msbm10}

\begin{document}
\title[The product of operators with closed range]
{The product of operators with closed range in Hilbert C*-modules}
\author{K. Sharifi}
\address{Kamran Sharifi, \newline Department of Mathematics,
Shahrood University of Technology, P. O. Box 3619995161-316,
Shahrood, Iran \newline
 School of Mathematics, Institute for Research in
 Fundamental Sciences (IPM), P.O. Box: 19395-5746, Tehran, Iran}
\email{sharifi.kamran@gmail.com and
sharifi@shahroodut.ac.ir}

\thanks{This research was in part supported by a
grant from IPM (No. 89460018).}


\begin{abstract}
Suppose $T$ and $S$ are bounded adjointable operators
with close range between Hilbert C*-modules, then $TS$ has closed
range if and only if $Ker(T)+Ran(S)$ is an orthogonal summand, if
and only if $Ker(S^*)+Ran(T^*)$ is an orthogonal summand. Moreover,
if the Dixmier (or minimal) angle between $Ran(S)$ and $Ker(T) \cap [Ker(T) \cap
Ran(S)]^{\perp}$ is positive and $ \overline{Ker(S^*)+Ran(T^* )}
$ is an orthogonal summand then $TS$ has closed range.
\end{abstract}
\maketitle

\section{Introduction.}
The closeness of range of operators is an attractive and
important problem which appears in operator theory, especially,
in the theory of Fredholm operators and generalized inverses. In
this paper we will investigate when the product of two operators
with closed range again has closed range. This problem was first
studied by Bouldin  for bounded operators between Hilbert spaces
in \cite{Bouldin1972, Bouldin1977}. Indeed, for Hilbert space
operators $T,S$ whose ranges are closed, he proved that the range
of $TS$ is closed if and only if the Dixmier (or minimal) angle between
$Ran(S)$ and $Ker(T) \cap [Ker(T) \cap Ran(S)]^{\perp}$ is
positive, where the Dixmier
angle between subspaces $M$ and $N$ of a certain Hilbert
space is the angle $ \alpha_0 (M,N)$ in $[0, \pi /2]$
whose cosine is defined by $c_0(M,N)= {\rm sup} \{  \| \langle x,y
\rangle \|  :  x \in M, \ \|x \| \leq 1 \, , \ y \in N, \ \|y \|
\leq 1 \}.$ Nikaido \cite{Nikaido1980, Nikaido1986} also gave
topological characterizations of the problem for the Banach space
operators. Recently (Dixmier and Friedrichs) angles between
linear subspaces have been studied systematically by Deutsch
\cite{Deutsch/angle}, he also has reconsidered the closeness of
range of the product of two operators with closed range. In this
note we use C*-algebras techniques to reformulate some results of
Bouldin and Deutsch in the framework of Hilbert C*-modules. Some
further characterizations of modular operators with closed range
are obtained.

Hilbert C*-modules are essentially objects like Hilbert spaces,
except that the inner product, instead of being complex-valued,
takes its values in a C*-algebra. Since the geometry of these
modules emerges from the C*-valued inner product, some basic
properties of Hilbert spaces like Pythagoras' equality,
self-duality, and decomposition into orthogonal complements do
not hold. The theory of Hilbert C*-modules, together with
adjointable operators forms an infrastructure for some of the
most important research topics in operator algebras, in Kasparov's
KK-theory and in noncommutative geometry.

A (left) {\it pre-Hilbert C*-module} over a C*-algebra
$\mathcal{A}$ is a left $\mathcal{A}$-module $E$ equipped with an
$\mathcal{A}$-valued inner product $\langle \cdot , \cdot \rangle
: E \times E \to \mathcal{A}\,, \ (x,y) \mapsto \langle x,y
\rangle$, which is $\mathcal A$-linear in the first variable $x$
(and conjugate-linear in $y$) and has the properties:
$$ \langle x,y \rangle=\langle y,x \rangle ^{*}, \ \langle ax,y \rangle
=a \langle x,y \rangle \, \ {\rm for} \ {\rm all} \ a \ {\rm in} \
\mathcal{A},$$ $$ \langle x,x \rangle \geq 0 \ \ {\rm with} \
   {\rm equality} \ {\rm only} \ {\rm when} \  x=0.$$

A pre-Hilbert $\mathcal{A}$-module $E$ is called a {\it Hilbert $
\mathcal{A}$-module} if $E$ is a Banach space with respect to the
norm $\| x \|=\|\langle x,x\rangle \| ^{1/2}$.  A Hilbert
$\mathcal A$-submodule $E$ of a Hilbert $\mathcal A$-module $F$
is an orthogonal summand if $F=E \oplus E^\bot $, where $E^\bot
:=\{ y \in F: ~ \langle x,y \rangle=0~ \ {\rm for} \ {\rm all}\ x
\in E \}$ denotes the orthogonal complement of $E$ in $F$. The
papers \cite{FR3, FR2} and the books \cite{LAN, M-T} are used as
standard sources of reference.

Throughout the present paper we assume $\mathcal{A}$ to be an
arbitrary C*-algebra (i.e. not necessarily unital). We use the
notations $Ker(\cdot)$ and $Ran(\cdot)$ for kernel and range of
operators, respectively. We denote by $\mathcal{L}(E,F)$ the
Banach  space of all bounded adjointable operators between $E$ and
$F$, i.e., all bounded $\mathcal A$-linear maps $T :E \rightarrow
F$ such that there exists $T^*:F \rightarrow E$ with the property
$ \langle Tx,y \rangle =\langle x,T^*y \rangle$ for all $ x \in
E$, $y \in F$. The C*-algebra $\mathcal{L}(E,E)$ is abbreviated
by $\mathcal{L}(E)$.

In this paper we first briefly investigate some basic facts about
Moore-Penrose inverses of bounded adjointable operators on
Hilbert C*-modules and then we give some necessary and sufficient
conditions for closeness of the range of the product of two
orthogonal projections. These lead us to our main results. Indeed,
for adjointable module maps $T,S$ whose ranges are closed we show
that the operator $TS$ has closed range if and only if
$Ker(T)+Ran(S)$ is an orthogonal summand, if and only if
$Ker(S^*)+Ran(T^*)$ is an orthogonal summand. The Dixmier angle
between submodules $M$ and $N$ of a Hilbert C*-module $E$ is the
angle $ \alpha_0 (M,N)$ in $[0, \pi /2]$ whose cosine is defined
by $$c_0(M,N)= {\rm sup} \{  \| \langle x,y \rangle \| : x \in M,
\ \|x \| \leq 1 \, , \ y \in N, \ \|y \| \leq 1 \}.$$  If the
Dixmier angle between $Ran(S)$ and $Ker(T) \cap [Ker(T) \cap
Ran(S)]^{\perp}$ is positive and $ \overline{Ker(S^*)+Ran(T^* )}
$ is an orthogonal summand then $TS$ has closed range. Since
every C*-algebra is a Hilbert C*-module over itself, our results
are also remarkable in the case of bounded adjointable operators
on C*-algebras.
\section{Preliminaries}

Closed submodules of Hilbert modules need not to be orthogonally
complemented at all, but Lance states in \cite[Theorem 3.2]{LAN}
under which conditions closed submodules may be orthogonally
complemented (see also \cite[Theorem 2.3.3]{M-T}). Let $E$, $F$ be
two Hilbert $\mathcal{A}$-modules and suppose that an operator
$T$ in $ \mathcal{L}(E,F)$ has closed range, then one has:
\begin{itemize}
\item $Ker(T)$ is orthogonally complemented in $E$, with
      complement $Ran(T^*)$,
\item $Ran(T)$ is orthogonally complemented in $F$, with
      complement $Ker(T^*)$,
\item the map $T^* \in \mathcal{L}(F,E)$ has closed range, too.
\end{itemize}
\begin{lemma}\label{closedT*T} Suppose $T \in  \mathcal{L}(E,F)$.
 The operator $T$ has closed range if and only if $T \,T^*$ has
closed range. In this case, $Ran(T)=Ran(T\, T^*)$.
\end{lemma}
\begin{proof} Suppose $T$ has closed range, the proof of Theorem 3.2
of \cite{LAN} indicates that $Ran(T\, T^*)$ is closed  and
$Ran(T)=Ran(T\, T^*)$.

Conversely, if $T\, T^*$ has closed range then $ F=Ran(T\, T^*)
\oplus Ker(T\, T^*) = Ran(T\, T^*) \oplus Ker(T^*) \subset Ran(T)
\oplus Ker(T^*) \subset F$ which implies $T$ has closed range.
\end{proof}

Let $T \in \mathcal{L}(E,F)$, then a bounded adjointable operator
$S \in \mathcal{L}(F,E)$ is called an {\it inner inverse} of $T$
if $T S T =T$. If $T \in \mathcal{L}(E,F)$ has an inner inverse
$S$ then the bounded adjointable operator $T^{ \times } = S T S$
in $ \mathcal{L}(F,E)$ satisfies
\begin{equation} \label{inner inverse}
T \, T^{ \times } T =T  \   \ {\rm and } \ \  T^{ \times } T \,
T^{ \times } = T.
\end{equation}
The bounded adjointable operator $T^{ \times }$ which satisfies
(\ref{inner inverse}) is called {\it generalized inverse} of $T$.
It is known that a bounded adjointable operator $T$ has a
generalized inverse if and only if $Ran(T)$ is closed, see e.g.
\cite{Hejazian/JMAA, Xu/common}.

Let $T \in \mathcal{L}(E,F)$, then a bounded adjointable operator
$T^{ \dag} \in \mathcal{L}(F,E)$ is called the {\it Moore-Penrose
inverse} of $T$ if
\begin{equation} \label{MP inverse}
T \, T^{ \dag}T=T, \ T^{ \dag}T \, T^{ \dag}= T^{ \dag}, \ (T \,
T^{ \dag})^*=T \, T^{ \dag} \ {\rm and} \ ( T^{ \dag} T)^*= T^{
\dag} T.
\end{equation}
The notation $T^{ \dag}$ is reserved to denote the Moore-Penrose
inverse of $T$. These properties imply that $T^{ \dag}$ is unique
and $ T^{ \dag} T$ and $ T \, T^{ \dag} $ are orthogonal
projections. Moreover, $Ran( T^{ \dag} )=Ran( T^{ \dag}  T)$,
$Ran( T )=Ran( T \, T^{ \dag})$,  $Ker(T)=Ker( T^{ \dag} T)$ and
$Ker(T^{ \dag})=Ker( T \, T^{ \dag} )$ which lead us to $ E= Ker(
T^{ \dag} T) \oplus Ran( T^{ \dag} T)= Ker(T) \oplus Ran( T^{
\dag} )$ and $F= Ker(T^{ \dag}) \oplus Ran(T).$

 Xu and Sheng in \cite{Xu/Sheng} have shown that a bounded
adjointable operator between two Hilbert C*-modules admits a
bounded Moore-Penrose inverse if and only if the operator has
closed range. The reader should be aware of the fact that a
bounded adjointable operator may admit an unbounded operator as
its Moore-Penrose, see \cite{FS2, SHA/PARTIAL, SHA/Groetsch} for
more detailed information.

\begin{proposition} \label{closed0} Suppose $E, F, G$ are Hilbert
$ \mathcal{A}$-modules and  $S \in \mathcal{L}(E,F)$ and
$T \in \mathcal{L}(F,G)$ are bounded adjointable operators with closed ranges.
Then $TS$ has a generalized inverse if and only if $ T^{ \dag} TS S^{ \dag}$ has.
In particular, $TS$ has closed range if and only if $ T^{ \dag} TS
S^{ \dag}$ has.
\end{proposition}

\begin{proof} Suppose first that $V$ is a generalized inverse
of $TS$. Then
\begin{eqnarray*}
T^{ \dag} T S S^{ \dag} (SVT) T^{ \dag} T S S^{ \dag} = T^{ \dag}
T \, ( S S^{ \dag} S) \, V \, (T \, T^{ \dag} T ) \, S S^{ \dag}
& = & T^{ \dag} T  S \, V \, T  S S^{ \dag}  \\
& = & T^{ \dag} TS S^{ \dag}.
\end{eqnarray*}
Similarly, $ SVT \, (T^{ \dag} T S S^{ \dag}) \, SVT = SVT $ and
so $SVT$ is a generalized inverse of $ T^{ \dag} TS S^{ \dag}$.
Conversely, suppose that $U \in \mathcal{L} (F) $ is a generalized
inverse of $ T^{ \dag} TS S^{ \dag}$. Let $P=S S^{ \dag}$ and
$Q=T^{ \dag} T$ are orthogonal projections onto $Ran(S)$ and
$Ker(T)^{ \perp}$, respectively, then $QPUQP=QP$. We set $W=PUQ$, then
$PWQ=W$ and $QWP=QP$. The later equality implies that
$Q(1-W)P=0$, that is, $1-W$ maps $Ran(P)=Ran(S)$ into
$Ker(Q)=Ker(T)$. Consequently, $T(1-W)S=0$. Hence,
$$TS \, (S^{ \dag} W T^{ \dag}) \,  TS = T P W Q S = TWS=TS .$$
On the other hand, $S^{ \dag} W T^{ \dag}= S^{ \dag} PUQ T^{
\dag}= S^{ \dag}S S^{ \dag} U T^{ \dag} T \, T^{ \dag}= S^{ \dag}
U T^{ \dag}$ which shows that $(S^{ \dag} W T^{ \dag}) \, TS \,
(S^{ \dag} W T^{ \dag}) = S^{ \dag} U T^{ \dag}= S^{ \dag} W T^{
\dag}$, i.e. $S^{ \dag} W T^{ \dag}$ is a generalized inverse of
$TS$. In particular, $TS$ has closed range if and only if $ T^{
\dag} TS S^{ \dag}$ has.
\end{proof}

\begin{lemma} \label{closed1}
Let $T \in  \mathcal{L}(E,F)$, then $T$ has closed range if and
only if $Ker(T)$ is orthogonally complemented in $E$ and $T$ is
bounded below on $Ker(T)^{\perp} $, i.e. $\| Tx \| \geq c \|x
\|$, for all $x \in Ker(T)^{\perp} $ for a certain positive
constant $c$.
\end{lemma}

The statement directly follows from Proposition 1.3 of \cite{F-S}.

\begin{lemma} \label{closed2}
Let $T$ be a non-zero bounded adjointable operator in $
\mathcal{L}(E,F)$, then $T$ has closed range if and only if
$Ker(T)$ is orthogonally complemented in $E$ and
$$ \gamma (T)= {\rm inf} \{ \| Tx \| : x \in Ker(T)^{\perp}
\ {\rm and} \  \|x \|=1 \} > 0.$$ In this case, $ \gamma (T)= \|
T ^{ \dag} \| ^{-1}$ and  $ \gamma (T) = \gamma (T^*)$.
\end{lemma}

\begin{proof} The first assertion follows directly from Lemma
\ref{closed1}. To prove the first equality, suppose $T$ has closed
range, $x \in Ker(T)^{\perp} = Ran( T^{ \dag} T)$ and $ \|x \|=1$,
then $ 1= \|x \|= \| T^{ \dag} T x \| \leq \| T^{ \dag} \| \, \|
Tx \|$, consequently, $ \| T ^{ \dag} \| ^{-1} \leq \gamma (T).$
Suppose $ x \in Ker(T)^{\perp}$ then $ \gamma (T) \|x \| \leq
\|Tx \| $. Suppose $w \in F$ and $ x = T^{ \dag} w$ then $ x \in
Ran( T^{ \dag})= Ker(T)^{\perp}$, hence,
$$  \gamma(T) \| T^{ \dag} w  \| \leq \|T \, T^{ \dag} w \| \leq
\| T \, T^{ \dag} \| \, \| w \| \leq \| w \|. $$ We therefore have
$ \gamma (T) \leq \| T ^{ \dag} \| ^{-1}.$ To establish the second
equality just recall that $T$ has closed range if and only if
$T^*$ has. It now follows from the first equality and the fact
$\|T^{* \, \dag} \|=\|T^{ \dag \, *} \|= \|T^{ \dag} \|.$
\end{proof}
\section{Closeness of the range of the products}

Suppose $F$ is a Hilbert $ \mathcal{A}$-module and $T$  be a
bounded adjointable operator in the unital C*-algebra $
\mathcal{L}(F)$, then $ \sigma(T)$ and ${\rm acc}~ \sigma(T)$
denote the spectrum and the set of all accumulation points of $
\sigma(T)$, respectively. According to \cite[Theorem
2.4]{KOL/Drazin1} and \cite[Theorem 2.2]{Xu/Sheng}, a bounded
adjointable operator $T$ in $ \mathcal{L}(F)$ has closed range if
and only if $T$ has a Moore-Penrose inverse, if and only if $ 0
\notin {\rm acc}~ \sigma(T \, T^*)$, if and only if $ 0 \notin
{\rm acc}~ \sigma(T^*T)$. In particular, if $T$ is selfadjoint
then $T$ has closed range if and only if $ 0 \notin {\rm acc}~
\sigma(T)$. We use these facts in the proof of the following
results.
\begin{lemma} \label{closed5-5}Suppose $F$ is a Hilbert
$ \mathcal{A}$-module and $P, Q$ are orthogonal projections in $
\mathcal{L}(F)$. Then $P-Q$ has closed range if and only if $P+Q$
has closed range.
\end{lemma}
\begin{proof} Following the argument of Koliha and Rako\v{c}evi\'c
\cite{Koliha/Rakocevic2005}, for every $ \lambda \in \mathbb{C}$
we have
\begin{equation} \label{Koliha(2.4)}
( \lambda -1 +P) ( \lambda -(P-Q)) ( \lambda -1 + Q) = \lambda(
\lambda^2 -1 + PQ),
\end{equation}
\begin{equation} \label{Koliha(2.5)}
( \lambda -1 +P) ( \lambda -(P+Q)) ( \lambda -1 +Q) = \lambda((
\lambda -1)^2 -PQ).
\end{equation}
Using the above equations and the facts that $ \sigma(P) \subset
\{0,1 \}$ and $ \sigma(Q) \subset \{0,1 \}$, we obtain that
$Ran(P-Q)$ is closed if and only if $0 \notin {\rm acc}~
\sigma(P-Q)$, if and only if $1 \notin {\rm acc}~ \sigma(PQ)$, if
and only if $0 \notin {\rm acc}~ \sigma(P+Q)$, if and only if
$Ran(P+Q)$ is closed.
\end{proof}
\begin{lemma} \label{closed6}Suppose $F$ is a Hilbert
$ \mathcal{A}$-module and $P, Q$ are orthogonal projections in $
\mathcal{L}(F)$. Then the following conditions are equivalent:
\newcounter{cou001}
\begin{list}{(\roman{cou001})}{\usecounter{cou001}}
\item $PQ$ has closed range,
\item $1-P-Q$ has closed range,
\item $1-P+Q$ has closed range,
\item $1-Q+P$ has closed range.
\end{list}
\end{lemma}
\begin{proof} Suppose $ \lambda \in
\mathbb{C} \setminus \{ 0,1 \}$. In view of the equation
(\ref{Koliha(2.5)}), we conclude that $ \lambda \in \sigma (P+Q)
$ if and only if $( \lambda-1)^2  \in \sigma (PQ) $.

The above fact together with Remark 1.2.1 of \cite{MUR} imply
that $PQ$ has closed range if and only if $0 \notin {\rm acc}~
\sigma(PQP)$, if and only if $0 \notin {\rm acc}~ \sigma(P^2Q)$,
if and only if $ 1 \notin {\rm acc}~ \sigma(P+Q)$, if and only if
$ 0 \notin {\rm acc}~ \sigma (1-P-Q)$, if and only if $1-P-Q$ has
closed range. This proves the equivalence of (i) and (ii). The
statements (ii), (iii) and (iv) are equivalent by Lemma
\ref{closed5-5}.
\end{proof}
\begin{remark} \label{closed7} Suppose $E$, $F$
are two Hilbert $ \mathcal{A}$-modules then the set of all
ordered pairs of elements $E \oplus F$ from $E$ and $F$ is a
Hilbert $\mathcal{A}$-module with respect to the $\mathcal{
A}$-valued inner product $\langle
(x_{1},y_{1}),(x_{2},y_{2})\rangle= \langle
x_{1},x_{2}\rangle_{E}+\langle y_{1},y_{2}\rangle _{F}$, cf.
\cite[Example 2.14]{R-W}. In particular, it can be easily seen
that $L$ is a closed submodule of $F$ if and only if $L \oplus \{
0 \}$ is a closed submodule of $F \oplus F$.
\end{remark}
\begin{lemma} \label{closed33}
Suppose $P$ and $Q$ are orthogonal projections on a Hilbert $
\mathcal{A}$-module $F$ then the following conditions are
equivalent:
\begin{list}{(\roman{cou001})}{\usecounter{cou001}}
\item $PQ$ has closed range,
\item $Ker(P)+Ran(Q)$ is an orthogonal summand,
\item $Ker(Q)+Ran(P)$ is an orthogonal summand.
\end{list}
\end{lemma}
\begin{proof}Suppose
$$ T= \begin{pmatrix}  1-P& Q \\ 0 & 0
\end{pmatrix} \in \mathcal{L}(F \oplus F).$$
Then $Ran(T)=(Ran(1-P) + Ran(Q)) \oplus \{0 \}$ and $Ran(T \, T
^*)=Ran(1-P+Q) \oplus \{ 0\}$. Using Lemmata \ref{closedT*T},
\ref{closed6} and Remark \ref{closed7}, we infer that $PQ$ has
closed range if and only if $1-P+Q$ has closed range, if and only
if $Ran(T \, T ^*)=Ran(1-P+Q) \oplus \{ 0\}$ is closed, if and
only if $Ran(T)=(Ran(1-P) + Ran(Q)) \oplus \{0 \}$ is closed, if
and only if $Ran(1-P) + Ran(Q)$ is closed. In particular,
$Ran(1-P+Q) = Ran(1-P) + Ran(Q)$ is an orthogonal summand. This
proves that the conditions (i) and (ii) are equivalent. Now,
consider the matrix operator
$$\tilde{T}= \begin{pmatrix}  1-Q & P \\ 0 & 0
\end{pmatrix} \in \mathcal{L}(F \oplus F).$$
A similar argument shows that $PQ$ has closed range if and only if
$Ran(1-Q+P) = Ran(1-Q) + Ran(P)$ is closed which shows that
conditions (i) and (iii) are equivalent.
\end{proof}
Suppose $M$ and $N$ are closed submodule of a Hilbert $
\mathcal{A}$-module $E$ and $P_{M}$ and $P_{N}$ are orthogonal
projection onto $M$ and $N$, respectively. Then $P_M \, P_N =
P_M$ if and only if $P_N \, P_M = P_M$, if and only if $ M
\subset N$. Beside these, the following statements are equivalent
\begin{itemize}
\item $P_{M}$ and $P_{N}$ commute, i.e. $P_M \, P_N=P_N \, P_M$,
\item $P_M \, P_N = P_{M \cap N}$,
\item $P_M \, P_N$ is an orthogonal projection,
\item $P_{M^{ \perp }}$ and $P_{N}$ commute,
\item $P_{N^{ \perp }}$ and $P_{M}$ commute,
\item $P_{M^{ \perp }}$ and $P_{N^{ \perp}}$ commute,
\item $M= M\cap N + M \cap N^{ \perp}$.
\end{itemize}

\begin{proposition} \label{closed3}
Suppose $P$ and $Q$ are orthogonal projections on a Hilbert $
\mathcal{A}$-module $F$ and $ \overline{Ker(Q)+Ran(P)} $ is an
orthogonal summand in $F$. If $R$ is the orthogonal projection
onto the closed submodule $ \overline{Ker(Q)+Ran(P)} $ and $PQ
\neq 0$ then
\begin{equation} \label{PQ1}
\gamma(PQ)^2 + \|(1-P)QR \|^2 \geq 1.
\end{equation}
\end{proposition}
\begin{proof}
The inclusion $ Ker(Q) \subset \overline{Ker(Q)+Ran(P)}$ implies
that the orthogonal projection $1-Q$ onto $Ker(Q)$ satisfies
$(1-Q)R=R(1-Q)=1-Q$, consequently, $QR$ is an orthogonal
projection and $Ran(QR)$ is orthogonally complemented in $F$.
Since $$Ran(QP) \subset Ran(QR) \subset \overline{Ran(QP)},$$ we
have $ \overline{Ran(QP)}= Ran(QR)$ and so $ \overline{Ran(QP)}$
is orthogonally complemented. Therefore, $Ker(PQ)^{ \, \perp}=
Ran(QR).$ Suppose $x \in Ker(PQ)^{ \perp} \subset Ran(Q)$ and
$\|x \|=1$. Then, since $x=QR \, x=Qx$, we have
\begin{eqnarray*}
\| PQ \,x \|^2 + \|(1-P)QR \|^2 & \geq & \| PQ \,x \|^2 + \|(1-P)Q
\,x \|^2  \\
& \geq &  \| \langle PQ \,x, PQ \,x \rangle + \langle (1-P)Q \,x,
(1-P)Q \,x \rangle \| \\
& = & \| \langle Qx,Qx \rangle \|= \| Qx \|^2=1.
\end{eqnarray*}
By definition, the infimum of $\| PQ \,x \|$ is $ \gamma (PQ)$.
Therefore, $ \gamma (PQ)^2 + \|(1-P)QR \|^2 \geq 1$.
\end{proof}

Note that as we set $\mathcal{A}= \mathbb{C}$ i.e. if we take
$F$ to be a Hilbert space, the inequality (\ref{PQ1})
changes to an equality. In view of this notification, the following problem
arises in the framework of Hilbert C*-modules.
\begin{problem}\label{problem-closed}
Suppose $P$ and $Q$ are orthogonal projections on a Hilbert $
\mathcal{A}$-module $F$ and $ \overline{Ker(Q)+Ran(P)} $ is an
orthogonal summand in $F$. If $R$ is the orthogonal projection
onto the closed submodule $ \overline{Ker(Q)+Ran(P)} $ and $PQ
\neq 0$ then characterize those
C*-algebras $ \mathcal{A}$ for which the following equality holds:
\begin{equation} \label{PQ1-1}
\gamma(PQ)^2 + \|(1-P)QR \|^2 =1.
\end{equation}
\end{problem}

To solve the problem, it might be useful to know that
$ \gamma (PQ) \leq \|PQ \, x \|$ for all
$x \in Ker(PQ)^{ \perp} \subset Ran(Q)$ of
norm $ \|x \|=1$, therefore $$ \gamma (PQ)^2 + \|(1-P)Q \, x \|^2   \leq  \| PQ \,x \|^2 +
\|(1-P)Q \,x \|^2  =  \| Px \|^2 + \|(1-P)x \|^2 .$$

\begin{corollary} \label{closed333}
Suppose $P$ and $Q$ are orthogonal projections on a Hilbert $
\mathcal{A}$-module $F$. If $ \delta = \|(1-P)QR \| < 1$ and $R$ is
the orthogonal projection onto the orthogonal summand
$\overline{Ker(Q)+Ran(P)}$ then $PQ$ has closed range.
\end{corollary}
\begin{proof}
Suppose $PQ \neq 0$ (in the case $PQ=0$ the result is clear).
According to Proposition \ref{closed3}
and its proof, $Ker(PQ)^{ \, \perp}=
Ran(QR)$ is orthogonally complimented and
$ \gamma(PQ)^2 \geq 1- \delta^{ \, 2} > 0$.
Therefore, $PQ$ has closed range by Lemma \ref{closed2}.
\end{proof}

Two different concepts of angle between subspaces of a Hilbert
space was first introduced by Dixmier and Friedrichs, see
\cite{Dixmier1949, Friedrichs, Arias-Gonzalez} and the excellent
survey by Deutsch \cite{Deutsch/angle} for more historical notes
and information. We generalized Dixmier's definition for the
angle between two submodules of a Hilbert C*-module.

\begin{definition} \label{closed4}
The Dixmier (or minimal) angle between submodules $M$ and $N$ of a
Hilbert C*-module $E$ is the angle $ \alpha_0 (M,N)$ in $[0, \pi
/2]$ whose cosine is defined by $$c_0(M,N)= {\rm sup} \{  \|
\langle x,y \rangle \|  :  x \in M, \ \|x \| \leq 1 \, , \ y \in
N, \ \|y \| \leq 1 \}.$$
\end{definition}

Suppose $M$ and $N$ are submodule of a Hilbert C*-module $E$, then
$(M+N)^{ \perp}=M^{ \perp} \cap N^{ \perp}$. In particular, if $
\overline{M+N}$ is orthogonally complemented in $E$ then
$$(M^{ \perp} \cap N^{ \perp})^{ \perp} = (M+N)^{ \perp \, \perp}= \overline{M+N}.$$

\begin{theorem} \label{closed5}
Suppose $S \in \mathcal{L}(E,F)$ and $T \in \mathcal{L}(F,G)$ are
bounded adjointable operators with closed range. Then the following
three conditions are equivalent:
\begin{list}{(\roman{cou001})}{\usecounter{cou001}}
\item $TS$ has closed range,
\item $Ker(T)+Ran(S)$ is an orthogonal summand in $F$,
\item $Ker(S^*)+Ran(T^*)$ is an orthogonal summand in $F$.
\end{list}
Furthermore, if $c_0 ( Ran(S), Ker(T) \cap [Ker(T) \cap
Ran(S)]^{\perp} )< 1$ and $\overline{Ker(S^*)+Ran(T^*)}$ is an
orthogonal summand then $TS$ has closed range.
\end{theorem}

\begin{proof}
Taking $P=T^{ \dag} \, T$ and $Q=SS^{ \dag}$, then
$$ Ker(P)=Ker(T)\,, \ Ran(P)=Ran(T^{ \dag})=Ran(T^*), $$
$$ Ker(Q)=Ker(S^{ \dag})= Ker(S^*)\,, \ {\rm and} \ \ Ran(Q)=Ran(S).$$
The equivalence of (i), (ii) and (iii) directly follows from the
above equalities and Lemma \ref{closed33}. To establish the
statement of the second part suppose $R$ is the orthogonal
projection onto the orthogonal summand $\overline{Ker(Q)+Ran(P)}$
then $(1-P)R$ is the projection onto
\begin{eqnarray*}
M=Ker(P) \cap [ \, \overline{ Ran(P) + Ker(Q) } \, ]&=& Ker(T)
\cap
[ \, \overline{ Ran(T^*) + Ker(S^*) } \,] \\
&=& Ker(T) \cap [ Ran(T^*)^{ \perp} \cap Ker(S^*)^{ \perp} \, ]^{  \perp} \\
&=& Ker(T) \cap [ Ker(T) \cap Ran(S) ]^{ \perp}.
\end{eqnarray*}
If neither $M$ nor $Ran(S)$ is $\{ 0 \}$, by commutativity of $R$
with $P$ and $Q$, we obtain

\begin{eqnarray*}
\|(1-P)QR \| & = & \| RQ(1-P) \| \\
& = & \| Q(1-P)R \| \\
& = & {\rm sup} \{ \| \langle Q(1-P)Rx, y  \rangle  \|: \,  x,y
\in F \ {\rm and} \  \|x \| \leq 1, \  \| y \| \leq 1 \} \\
& = & {\rm sup} \{ \| \langle (1-P)Rx, Qy  \rangle  \|: \,  x,y
\in F \ {\rm and} \  \|x \| \leq 1, \  \| y \| \leq 1 \} \\
& = & {\rm sup} \{ \| \langle x,y  \rangle  \|: \,  x \in M, \, y
\in Ran(S) \ {\rm and} \  \|x \| \leq 1, \  \| y \| \leq 1 \}\\
& = & c_0(M,Ran(S)).
\end{eqnarray*}
The statement is now derived from the above
argument and Corollary \ref{closed333}.
\end{proof}

Recall that a bounded adjointable operator between Hilbert
C*-modules admits a bounded adjointable Moore-Penrose inverse if
and only if the operator has closed range. This lead us to the
following results.

\begin{corollary} \label{closed7-8}Suppose $S \in \mathcal{L}(E,F)$
and $T \in \mathcal{L}(F,G)$ possess bounded adjointable
Moore-Penrose inverses $S^{ \dag}$ and $T^{ \dag}$. Then $(TS)^{
\dag}$ is bounded if and only if $Ker(T)+Ran(S)$ is an orthogonal
summand, if and only if $Ker(S^*)+Ran(T^*)$ is an orthogonal
summand. Moreover,
if the Dixmier angle between $Ran(S)$ and $Ker(T) \cap [Ker(T)
\cap Ran(S)]^{\perp}$ is positive and
$\overline{Ker(S^*)+Ran(T^*)}$ is an orthogonal summand then
$(TS)^{ \dag}$ is bounded.
\end{corollary}

Now, it is natural to ask for the reverse order law, that is, if
$S \in \mathcal{L}(E,F)$ and $T \in \mathcal{L}(F,G)$ possess
bounded adjointable Moore-Penrose inverses $S^{ \dag}$ and $T^{
\dag}$, when does the equation $(TS)^{ \dag}=S^{ \dag} \, T^{
\dag}$ hold? We will answer this question elsewhere. Note that
the above conditions do not ensure the equality.

Recall that a C*-algebra of compact operators is a $c_{0}$-direct
sum of elementary C*-algebras $\mathcal{K}(H_{i})$ of all compact
operators acting on Hilbert spaces $H_{i}, \ i \in I$, i.e.
$\mathcal{A} = c_{0}$-$ \oplus_{i \in I}\mathcal{K} (H_{i})$,
cf.~\cite[Theorem 1.4.5]{ARV}. Suppose $ \mathcal{A}$ is an
arbitrary C*-algebra of compact operators. Magajna and Schweizer
have shown, respectively, that every norm closed (coinciding with
its biorthogonal complement, respectively) submodule of every
Hilbert $ \mathcal{A}$-module is automatically an orthogonal
summand, cf. \cite{MAG, SCH}. In this situation, every bounded
$\mathcal{A}$-linear map $T:E \to F$ is automatically
adjointable. Recently further generic properties of the category
of Hilbert C*-modules over C*-algebras which characterize
precisely the C*-algebras of compact operators have been found in
\cite{FR1, F-S, FS2}. We close the paper with the observation
that we can reformulate Theorem \ref{closed5} in terms of bounded
$\mathcal{A}$-linear maps on Hilbert C*-modules over C*-algebras
of compact operators.
\begin{corollary} \label{closed8} Suppose $ \mathcal{A}$ is an
arbitrary C*-algebra of compact operators, $E, F, G$ are Hilbert
$ \mathcal{A}$-modules and $S: E \to F$ and $T :F \to G$ are
bounded $\mathcal{A}$-linear maps with close range. Then the
following conditions are equivalent:
\begin{list}{(\roman{cou001})}{\usecounter{cou001}}
\item $TS$ has closed range,
\item $Ker(T)+Ran(S)$ is closed,
\item $Ker(S^*)+Ran(T^*)$ is closed.
\end{list}
Furthermore, if $c_0 ( Ran(S), Ker(T) \cap [Ker(T) \cap
Ran(S)]^{\perp} )< 1$ then $TS$ has closed range.
\end{corollary}

In view of Corollary \ref{closed8}, one may ask about the converse of the
last conclusion. To find a solution, one way reader has is to solve
Problem \ref{problem-closed}.

{\bf Acknowledgement}: The author would like to thank
the referee for his/her careful reading and useful comments.

\end{document}